\documentclass{article}
\usepackage[T1]{fontenc}
\usepackage{lmodern}
\usepackage{geometry}
\usepackage{microtype}
\usepackage{mathtools}
\usepackage{amsfonts}
\usepackage{amssymb}
\DeclareMathOperator{\arccot}{arccot}
\usepackage{amsthm}
\newtheorem{lem}{Lemma}
\newtheorem{thm}{Theorem}
\title{Mode Stability for Gravitational Instantons of Type D}
\author{Gustav Nilsson\footnote{\texttt{gustav.nilsson@aei.mpg.de}}\\Max Planck Institute for Gravitational Physics (Albert Einstein Institute)\\Am Mühlenberg 1, D-14476 Potsdam, Germany}
\date{}
\begin{document}
\maketitle
\begin{abstract}\noindent
We study Ricci-flat perturbations of gravitational instantons of Petrov type D\@. Analogously to the Lorentzian case, the Weyl curvature scalars of extreme spin weight satisfy a Riemannian version of the separable Teukolsky equation. As a step toward infinitesimal rigidity of the type D Kerr and Taub-bolt families of instantons, we prove mode stability, i.e.,\ that the Teukolsky equation admits no solutions compatible with regularity and asymptotic (local) flatness.
\end{abstract}
\section{Introduction}
A gravitational instanton is a complete and non-compact Ricci-flat Riemannian four-manifold with quadratic curvature decay. There are a number of families of known examples, such as the Riemannian Kerr instanton and the Taub--NUT\footnote{The acronym ``NUT'' is formed by the initials of E. Newman, L. Tamburino, and T. Unti.} and Taub-bolt instantons. Furthermore, there are some known general results about gravitational instantons, many of which hold under various symmetry assumptions, such as the existence of a $U(1)$ or $U(1)\times U(1)$ isometry group, see e.g.\ \cite{aksteiner2023gravitational,nilsson2023topology,Biquard_2022,aksteiner2022gravitational}. In the compact case, we have the Besse conjecture \cite{MR867684}, stating that all compact Ricci-flat manifolds have special holonomy. This is a wide-open conjecture; there are no known examples of compact Ricci-flat four-manifolds with generic holonomy, i.e.,\ holonomy group $\mathit{SO}(4)$. This is in contrast to the non-compact case since there are examples of gravitational instantons with generic holonomy. However, all known examples still satisfy the weaker requirement of hermiticity, and it is, therefore, natural to conjecture that all gravitational instantons are Hermitian, see \cite{aksteiner2023gravitational} where this conjecture is stated for the ALF case. A first step toward such a result is given by proving \emph{rigidity}, i.e.,\ for various known examples of gravitational instantons, showing that there are no other Ricci-flat metrics close to that metric.

In the recent paper \cite{biquard2023gravitational}, it was shown that rigidity holds for the Riemannian Kerr and Taub-bolt families. However, \emph{infinitesimal rigidity}, i.e.,\ that Ricci-flat linear perturbations of these instantons decaying sufficiently fast at infinity must be perturbations within the respective family, is still open. Due to the fact that rigidity holds and that the relevant families are smooth manifolds, infinitesimal rigidity is equivalent to \emph{integrability}, i.e.,\ that any such Ricci-flat linear perturbation integrates to a curve of Ricci-flat metrics.

It was shown in \cite{MR0995773}, that for perturbations of the Lorentzian Kerr metric whose frequency lies in the upper half plane, and satisfying certain boundary conditions, the perturbations of the Weyl scalars of extreme spin weight vanish identically. This result is known as \emph{mode stability}. Furthermore, mode stability for frequencies on the real axis was shown in \cite{Andersson_2017}. As was shown in \cite{10.1063/1.1666203}, see also \cite{andersson2022mode}, perturbations of the Lorentzian Kerr metric whose Weyl scalars of extreme spin weight vanish identically must be perturbations within the Kerr family, modulo gauge.

The following two main theorems of this paper show that the Riemannian analog of mode stability holds in the ALF type D case,\footnote{By type D, we mean Petrov type $\mathrm{D}^+\mathrm{D}^-$, cf.\ \cite{Biquard_2022,aksteiner2022gravitational}. The only ALF instantons of type D are the Riemannian Kerr and the Taub-bolt metrics.} the precise definitions of asymptotically flat (AF) and ALF perturbations being given in Sections \ref{sec:riemannian_kerr_perturbation_equation} and \ref{sec:taub_bolt_perturbation_equation}, respectively.

\begin{thm}
For Ricci-flat AF perturbations of the Riemannian Kerr metric, the perturbed Weyl scalars $\dot{\Psi}_0,\dot{\tilde{\Psi}}_0$ vanish identically.\label{thm:riemannian_kerr_mode_stability}
\end{thm}
\begin{thm}
For Ricci-flat ALF perturbations of the Taub-bolt metric, the perturbed Weyl scalars $\dot{\Psi}_0,\dot{\tilde{\Psi}}_0$ vanish identically.\label{thm:taub_bolt_mode_stability}
\end{thm}
Based on these results, one might conjecture that infinitesimal rigidity in the previously described sense holds for these instantons.

\subsection*{Acknowledgements}
The author would like to thank Lars Andersson, Mattias Dahl, Oliver Petersen, and Klaus Kröncke for their helpful comments and discussion. Special thanks should be given to Steffen Aksteiner for assistance with computations using the computer algebra package xAct for Mathematica\texttrademark, and for assistance with understanding the Newman--Penrose formalism. For the latter, the author would also like to give a special thanks to Bernardo Araneda.  This research was supported by the IMPRS for Mathematical and Physical Aspects of Gravitation, Cosmology and Quantum Field Theory.

\section{The Newman--Penrose Formalism in Riemannian Signature}\label{sec:riemannian_NP_formalism}
The Newman--Penrose (NP) formalism \cite{MR141500}, commonly used in general relativity, can be adapted to a Riemannian signature (see \cite{MR1669188,MR1294497}). Let $(l,\overline{l},m,\overline{m})$ be a tetrad of vector fields with complex coefficients, in which the metric has the form \begin{equation}\begin{pmatrix}0&1&0&0\\1&0&0&0\\0&0&0&1\\0&0&1&0\end{pmatrix}.\end{equation} When viewed as first order differential operators, we denote\footnote{Note that the symbol $\Delta$ will be given different, unrelated meanings in Section \ref{sec:riemannian_kerr} and Section \ref{sec:taub_bolt}. The current meaning of $\Delta$ will be retained throughout all of Section \ref{sec:riemannian_NP_formalism} and Appendix \ref{sec:NP_equations}.} the vector fields $l,\overline{l},m,\overline{m}$ by $D,\Delta ,\delta,-\tilde{\delta}$, respectively. With respect to the tetrad $(l,\overline{l},m,\overline{m})$, the Levi-Civita connection is represented by $24$ \emph{spin coefficients}, denoted by Greek letters and defined to be the coefficients in the right-hand sides of the equations
\begin{align}
\frac{1}{2}(\overline{l}^a\nabla_bl_a-m^a\nabla_b\overline{m}_a) & =\gamma l_b+\epsilon\overline{l}_b-\alpha m_b+\beta\overline{m}_b,                           \\
\overline{l}^a\nabla_b\overline{m}_a                             & =-\nu l_b-\pi\overline{l}_b+\lambda m_b-\mu\overline{m}_b,                                   \\
m^a\nabla_bl_a                                                   & =\tau l_b+\kappa\overline{l}_b-\rho m_b+\sigma\overline{m}_b,                                \\
\frac{1}{2}(\overline{l}^a\nabla_bl_a+\overline{m}^a\nabla_bm_a) & =\tilde{\gamma}l_b+\tilde{\epsilon}l_b-\tilde{\beta}m_b+\tilde{\alpha}\overline{m}_b,        \\
\overline{l}^a\nabla_bm_a                                        & =\tilde{\nu}l_b+\tilde{\pi}\overline{l}_b-\tilde{\mu}m_b+\tilde{\lambda}\overline{m}_b,      \\
\overline{m}^a\nabla_bl_a                                        & =-\tilde{\tau}l_b-\tilde{\kappa}\overline{l}_b+\tilde{\sigma}m_b-\tilde{\rho}\overline{m}_b. 
\end{align}
From the fact that all inner products of the tetrad vectors are constant, it can also be seen that
\begin{align}
\overline{\alpha}       & =\beta,       & \overline{\gamma}       & =-\epsilon,       & \overline{\kappa}       & =\nu,       & \overline{\lambda}       & =-\sigma,       & \overline{\mu}       & =-\rho,       & \overline{\pi}       & =\tau,       \\
\overline{\tilde\alpha} & =\tilde\beta, & \overline{\tilde\gamma} & =-\tilde\epsilon, & \overline{\tilde\kappa} & =\tilde\nu, & \overline{\tilde\lambda} & =-\tilde\sigma, & \overline{\tilde\mu} & =-\tilde\rho, & \overline{\tilde\pi} & =\tilde\tau. 
\end{align}
The Levi-Civita connection is thus represented by $6+6$ independent complex scalars.

We also have the Weyl scalars:
\begin{gather}
\begin{aligned}
\Psi_0 & =-W(l,m,l,m),                                 &   
\Psi_1 & =-W(l,\overline{l},l,m),                      &   
\Psi_2&=W(l,m,\overline{l},\overline{m}),\\
\Psi_3 & =W(l,\overline{l},\overline{l},\overline{m}), &   
\Psi_4&=-W(\overline{l},\overline{m},\overline{l},\overline{m}),
\end{aligned}\\
\begin{aligned}
\tilde{\Psi}_0 & =-W(l,\overline{m},l,\overline{m}), &   
\tilde{\Psi}_1 & =-W(l,\overline{l},l,\overline{m}), &   
\tilde{\Psi}_2&=W(l,\overline{m},\overline{l},m),\\
\tilde{\Psi}_3 & =W(l,\overline{l},\overline{l},m),  &   
\tilde{\Psi}_4&=-W(\overline{l},m,\overline{l},m),
\end{aligned}
\end{gather}
where $W$ denotes the Weyl curvature tensor.

From the definitions, one sees immediately that $\overline{\Psi}_k=\Psi_{4-k}$ and $\overline{\tilde{\Psi}}_k=\tilde{\Psi}_{4-k}$, so that the Weyl tensor is determined by the $2+2$ complex scalars $\Psi_0,\Psi_1,\tilde{\Psi}_0,\tilde{\Psi}_1$ and the $1+1$ real scalars $\Psi_2,\tilde{\Psi}_2$. The Weyl scalars of \emph{extreme spin weight} are defined to be $\Psi_0,\Psi_4,\tilde{\Psi}_0$ and $\tilde{\Psi}_4$.

A tetrad $(l,\overline{l},m,\overline{m})$ is said\footnote{The corresponding notion in the Lorentzian setting is that of a \emph{principal} tetrad.} to be \emph{adapted} if $\Psi_0=\Psi_1=\tilde{\Psi}_0=\tilde{\Psi}_1=0$ and $\Psi_2,\tilde{\Psi}_2\neq0$. It can be shown that a Ricci-flat four-manifold admits an adapted tetrad if and only if it has type D\@. A proof of this fact in the Lorentzian case can be found in \cite[Chapter~7]{MR838301}.
\subsection{The Perturbation Equations}
When referring to a \emph{perturbation} $\dot{g}$ of a metric $g$, we are referring to a linear perturbation of $g$, i.e.\ a symmetric two-tensor $\dot{g}$. When $g$ is Ricci-flat, we say that $\dot{g}$ is a \emph{Ricci-flat perturbation} if it is Ricci-flat to first order, i.e.\ if $\dot{g}\in\ker((D\operatorname{Ric})_g)$. In general, for a quantity depending on the metric $g$, we let a dot above the quantity denote its derivative in the direction $\dot{g}$. Then $\dot{g}$ is a Ricci-flat perturbation if and only if $\dot{\operatorname{Ric}}=0$.

Ricci-flat perturbations of the Lorentzian Kerr metric have been studied extensively, and in \cite{1973ApJ...185..635T}, Teukolsky derived a well-known equation for the perturbation of the Weyl scalars of extreme spin weight, for such perturbations of the metric. The following theorem gives a Riemannian analog of that perturbation equation.

\begin{thm}
Consider a Ricci-flat perturbation $\dot{g}$ of a Ricci-flat type D metric $g$. For an adapted tetrad, the perturbation $\dot{\Psi}_0$ satisfies the equation
\begin{equation}
((D-3\epsilon+\tilde{\epsilon}-\tilde{\rho}-4\rho)(\Delta-4\gamma+\mu)-(\delta-\tilde{\alpha}-3\beta+\tilde{\pi}-4\tau)(\tilde{\delta}-4\alpha+\pi)-3\Psi_2)\dot{\Psi}_0=0,\label{eq:riemannian_teukolsky_1}
\end{equation}
and the perturbation $\dot{\tilde{\Psi}}_0$ satisfies the equation
\begin{equation}((D-3\tilde{\epsilon}+\epsilon-\rho-4\tilde{\rho})(\Delta -4\tilde{\gamma}+\tilde{\mu})-(\tilde{\delta}-\alpha-3\tilde{\beta}+\pi-4\tilde{\tau})(\delta-4\tilde{\alpha}+\tilde{\pi})-3\tilde{\Psi}_2)\dot{\tilde{\Psi}}_0=0.\label{eq:riemannian_teukolsky_2}\end{equation}
\end{thm}
\begin{proof}
Since we have an adapted tetrad, $\Psi_0=\Psi_1=\tilde{\Psi}_0=\tilde{\Psi}_1=0$, and using \eqref{eq:deltatPsi0} and \eqref{eq:DeltaPsi0} along with their tilded versions, we also have $\kappa=\tilde{\kappa}=\sigma=\tilde{\sigma}=0$. The linearized versions of \eqref{eq:DeltaPsi0} and \eqref{eq:deltatPsi0} become
\begin{equation}(\Delta-4\gamma+\mu)\dot{\Psi}_0=(\delta-4\tau-2\beta)\dot{\Psi}_1+3\dot{\sigma}\Psi_2\label{eq:Delta_Psi0_dot}\end{equation}
and
\begin{equation}(\tilde{\delta}-4\alpha+\pi)\dot{\Psi}_0=(D-4\rho-2\epsilon)\dot{\Psi}_1+3\dot{\kappa}\Psi_2\label{eq:deltat_Psi0_dot}\end{equation}
respectively. Operating on \eqref{eq:Delta_Psi0_dot} with $D$ and on $\eqref{eq:deltat_Psi0_dot}$ with $\delta$, subtracting the resulting equations and using the commutation relation \eqref{eq:commutator_D_delta}, we get
\begin{align}
(D(\Delta-4\gamma+\mu)-\delta(\tilde{\delta}-4\alpha+\pi))\dot{\Psi}_0 & = ([D,\delta]-4D\tau-2D\beta+4\delta\rho+2\delta\epsilon)\dot{\Psi}_1+(3D\dot{\sigma}-3\delta\dot{\kappa})\Psi_2 \\
                                                                       & =                                                                                                                
\begin{multlined}[t]
(-(\tilde{\alpha} + 3\beta -  \tilde{\pi}+4\tau) D+  (3\epsilon -  \tilde{\epsilon} + \tilde{\rho}+4\rho)) \delta)\dot{\Psi}_1 \\
-\dot{\Psi}_1(D(4\tau+2\beta)+\delta(4\rho+2\epsilon))+(3D\dot{\sigma}-3\delta\dot{\kappa})\Psi_2.
\end{multlined}
\end{align} 
We eliminate the first term on the right:
\begin{equation}
((D-3\epsilon+\tilde{\epsilon}-\tilde{\rho}-4\rho)(\Delta-4\gamma+\mu)-(\delta-\tilde{\alpha}-3\beta+\tilde{\pi}-4\tau)(\tilde{\delta}-4\alpha+\pi))\dot{\Psi}_0=A_1+A_2,
\end{equation}
where
\begin{multline}
A_1=((-3\epsilon+\tilde{\epsilon}-\tilde{\rho}-4\rho)(-4\tau-2\beta)-(-\tilde{\alpha}-3\beta+\tilde{\pi}-4\tau)(-4\rho-2\epsilon)\\
-\dot{\Psi}_1(D(4\tau+2\beta)+\delta(4\rho+2\epsilon))
\end{multline}
and
\begin{equation}
A_2=3((D-3\epsilon+\tilde{\epsilon}-\tilde{\rho}-4\rho)\dot{\sigma}-(\delta-\tilde{\alpha}-3\beta+\tilde{\pi}-4\tau)\dot{\kappa})\Psi_2.
\end{equation}
By using \eqref{eq:D_beta}, \eqref{eq:D_tau} and \eqref{eq:delta_rho}, we see that $A_1=0$. For an adapted tetrad, \eqref{eq:deltat_Psi1} and \eqref{eq:Delta_Psi1} become \begin{equation}D\Psi_2=3\rho\Psi_2,\qquad\delta\Psi_2=3\tau\Psi_2.\end{equation}
Therefore, by the Leibniz rule,
\begin{align}
A_2 & =3\dot{\sigma}D\Psi_2-3\dot{\kappa}\delta\Psi_2+3\Psi_2((D-3\epsilon+\tilde{\epsilon}-\tilde{\rho}-4\rho)\dot{\sigma}-(\delta-\tilde{\alpha}-3\beta+\tilde{\pi}-4\tau)\dot{\kappa})) \\
    & =3\Psi_2((D-3\epsilon+\tilde{\epsilon}-\tilde{\rho}-\rho)\dot{\sigma}-(\delta-\tilde{\alpha}-3\beta+\tilde{\pi}-\tau)\dot{\kappa})                                                   \\
    & =3\dot{\Psi}_0\Psi_2,                                                                                                                                                                
\end{align}
where we used the linearization of \eqref{eq:D_sigma} in the last step, showing that \eqref{eq:riemannian_teukolsky_1} holds. The proof of \eqref{eq:riemannian_teukolsky_2} is similar, referring to the tilded versions of the NP equations instead.
\end{proof}
\section{The Riemannian Kerr Instanton}\label{sec:riemannian_kerr}
In Boyer--Lindquist coordinates $(t,r,\theta,\phi)$, the Riemannian Kerr family of metrics is given by the expression\footnote{Note that $\Delta$ has a different meaning in this section, unrelated to those given in Section \ref{sec:riemannian_NP_formalism} and Section \ref{sec:taub_bolt}. The current meaning of $\Delta$ will be retained throughout all of Section \ref{sec:riemannian_kerr} and Appendix \ref{sec:riemannian_kerr_spin_coefficients}.} \begin{equation}g=\frac{\Sigma}{\Delta}\,dr^2+\Sigma\,d\theta^2+\frac{\Delta}{\Sigma}(dt-a\sin^2\theta\,d\phi)^2+\frac{\sin^2\theta}{\Sigma}((r^2-a^2)\,d\phi+a\,dt)^2.\label{eq:riemannian_kerr_metric}\end{equation} Here, $M>0$ and $a\in\mathbb{R}$ are the parameters of the family, $\Delta=\Delta(r)=r^2-2Mr-a^2$ and $\Sigma=r^2-a^2\cos^2\theta$, and the coordinates have the ranges $r>r_+$, $0<\theta<\pi$, where $r_\pm=M\pm\sqrt{M^2+a^2}$ are the roots of $\Delta$. Like its Lorentzian counterpart, this metric is Ricci-flat.

Now define new coordinates $(\tilde{t},\tilde{r},\theta,\tilde{\phi})$ by
\begin{equation}\begin{cases}r&=M+\sqrt{M^2+a^2}\cosh\tilde{r},\\t&=\frac{1}{\kappa}\tilde{t},\\\phi&=\tilde{\phi}-\frac{\Omega}{\kappa}\tilde{t},\end{cases}\end{equation} where $\kappa=\frac{\sqrt{M^2+a^2}}{2Mr_+}$ and $\Omega=\frac{a}{2Mr_+}$. Then $r$ is a smooth function of $\tilde{r}^2$, and \eqref{eq:riemannian_kerr_metric} gives \begin{equation}g=\Sigma(d\tilde{r}^2+d\theta^2+(\tilde{r}^2+O(\tilde{r}^4))\,d\tilde{t}^2+(\sin^2\theta+O(\sin^4\theta))\,d\tilde{\phi}^2).\end{equation} Letting $(\tilde{r},\tilde{t})$ be polar coordinates on $\mathbb{R}^2$ and letting $(\theta,\tilde{\phi})$ be spherical coordinates on $S^2$, it follows that $g$ extends to a complete metric on $\mathbb{R}^2\times S^2$, provided that we identify $\tilde{t}$ and $\tilde{\phi}$ with period $2\pi$ independently. Note that this is equivalent to performing the identifications $(t,\phi)\sim(t+\frac{2\pi}{\kappa},\phi-\frac{2\pi\Omega}{\kappa})\sim(t,\phi+2\pi)$.
\subsection{The Separated Perturbation Equations in Coordinates}\label{sec:riemannian_kerr_perturbation_equation}
We shall be interested in a particular choice of complex null tetrad $(l,\overline{l},m,\overline{m})$, called the \emph{Carter tetrad}, defined by
\begin{align}
l & =\frac{1}{\sqrt{2\Delta\Sigma}}\left((r^2-a^2)\frac{\partial}{\partial t}-a\frac{\partial}{\partial\phi}\right)+i\sqrt{\frac{\Delta}{2\Sigma}}\frac{\partial}{\partial r},              \\
m & =\frac{1}{\sqrt{2\Sigma}}\frac{\partial}{\partial\theta}-\frac{i}{\sqrt{2\Sigma}}\left(\frac{1}{\sin\theta}\frac{\partial}{\partial\phi}+a\sin\theta\frac{\partial}{\partial t}\right). 
\end{align}
Note that $|l|_g=|m|_g=1$. The spin coefficients for the Carter tetrad are given explicitly in Section \ref{sec:riemannian_kerr_spin_coefficients}. For this tetrad, we have
\begin{equation}
\Psi_2=\frac{M}{(r-a\cos\theta)^3},
\qquad
\tilde{\Psi}_2=\frac{M}{(r+a\cos\theta)^3},
\end{equation}
and all other Weyl scalars vanish. In particular, this is an adapted tetrad.

We shall now analyze the perturbation equations in the Carter tetrad. The relevant properties of the equations are given in the following four lemmas.
\begin{lem}
For the Carter tetrad, the perturbation equation \eqref{eq:riemannian_teukolsky_1} is equivalent to the equation\footnote{Note that the equation $\mathbf{L}\Phi=0$ is the same equation as that occurring in the Lorentzian case (see \cite{MR0995773,Andersson_2017}), but with $t$ replaced with $it$, $a$ replaced with $-ia$ and with $s=-2$.} $\mathbf{L}\Phi=0$, where $\Phi=\Psi_2^{-2/3}\dot{\Psi}_0$ and \begin{multline}\mathbf{L}=\frac{\partial}{\partial r}\Delta\frac{\partial}{\partial r}+\frac{1}{\Delta}\left((r^2-a^2)\frac{\partial}{\partial t}-a\frac{\partial}{\partial\phi}+2i(r-M)\right)^2+8i(r+a\cos\theta)\frac{\partial}{\partial t}\\+\frac{1}{\sin\theta}\frac{\partial}{\partial\theta}\sin\theta\frac{\partial}{\partial\theta}+\frac{1}{\sin^2\theta}\left(a\sin^2\theta\frac{\partial}{\partial t}+\frac{\partial}{\partial\phi}-2i\cos\theta\right)^2.\end{multline}

Furthermore, if $\Phi$ is a solution to this equation coming from a perturbation of the metric, then we can write
\begin{equation}
\Phi(t,r,\theta,\phi)=\sum_{m,\omega,\Lambda}e^{i(m\phi-\omega t)}R_{m,\omega,\Lambda}(r)S_{m,\omega,\Lambda}(\theta),\label{eq:riemannian_kerr_separation_ansatz}
\end{equation}
where $m$ runs over $\mathbb{Z}$, $\omega$ runs over $\Omega+\kappa\mathbb{Z}$ and for each choice of $m,\omega,\Lambda$, the function $R=R_{m,\omega,\Lambda}$ solves the equation $\mathbf{R}R=0$. The function $S=S_{m,\omega,\Lambda}$ is the unique solution to the boundary value problem $\mathbf{S}S=0$, $S'(0)=S'(\pi)=0$, where \begin{equation}\mathbf{R}=\frac{d}{dr}\Delta\frac{d}{dr}+U(r),\label{eq:riemannian_kerr_radial_equation}\end{equation}\begin{equation}U(r)=-\frac{((r^2-a^2)\omega+am+2(r-M))^2}{\Delta}+8r\omega-\Lambda\end{equation} and \begin{equation}\mathbf{S}=\frac{1}{\sin\theta}\frac{d}{d\theta}\sin\theta\frac{d}{d\theta}+V(\cos\theta),\end{equation}\begin{equation}V(x)=8a\omega x-\frac{1}{1-x^2}\left(a\omega(1-x^2)-m+2x\right)^2+\Lambda.\end{equation} Here, $S$ is normalized with respect to the $L^2$ product with measure $\sin\theta\,d\theta$, and the separation constant $\Lambda$ runs over the (countable set of) values for which such an $S$ exists.

The same statement holds for the perturbation equation \eqref{eq:riemannian_teukolsky_2}, if $\Phi$ is replaced by $\tilde{\Phi}$, where $\tilde{\Phi}=\tilde{\Psi}_2^{-2/3}\dot{\tilde{\Psi}}_0$.
\end{lem}
\begin{proof}
The fact that \eqref{eq:riemannian_teukolsky_1} is equivalent to $\mathbf{L}\Phi=0$ follows from a direct computation, using the expressions for the spin coefficients in Section \ref{sec:riemannian_kerr_spin_coefficients}.

Now note that the boundary value problem $\mathbf{S}S=0$, $S'(0)=S'(\pi)=0$ is a Sturm-Liouville problem. Thus, there exists an orthonormal $L^2$ basis of functions $\{S_{m,\omega,\Lambda}\}_\Lambda$ solving it, and furthermore, we can perform a Fourier series decompositions in the coordinates $(t,\phi)$. From these considerations, we can write \eqref{eq:riemannian_kerr_separation_ansatz}, where
\begin{equation}
R_{m,\omega,\Lambda}(r)=\frac{\kappa}{4\pi^2}\int_0^{2\pi/\kappa}\int_0^{2\pi}\int_0^\pi e^{-i(m\phi-\omega t)}\Phi(t,r,\theta,\phi) S_{m,\omega,\Lambda}(\theta)\sin\theta\,d\theta\,d\phi\,dt.\label{eq:riemannian_kerr_separation_integral}
\end{equation}
The fact that $R=R_{m,\omega,\Lambda}$ satisfies $\mathbf{R}R=0$ now follows directly from \eqref{eq:riemannian_kerr_separation_integral}, along with the fact that $\mathbf{L}\Phi=0$.

For the statement involving $\tilde{\Phi}$, the proof is entirely analogous.
\end{proof}
\begin{lem}
The equation $\mathbf{R}R=0$ is an ordinary differential equation in a complex variable $r$, which has regular singular points at $r=r_\pm$. The point $r=\infty$ is an irregular singular point of rank $1$, except when $\omega=0$, in which case it is a regular singular point. Thus, the equation $\mathbf{R}R=0$ is a confluent Heun equation (see \cite[Section~3]{MR1858237}) when $\omega\neq 0$, and a hypergeometric equation (see \cite[Section~2]{MR1858237}) when $\omega=0$. The characteristic exponents at $r=r_+$ are
\begin{equation}
\pm\left(1+\frac{2Mr_++am}{r_+-r_-}\right),
\end{equation}
and those at $r=r_-$ are
\begin{equation}
\pm\left(-1+\frac{2Mr_-+am}{r_+-r_-}\right).
\end{equation}
When $\omega=0$, we have $\Lambda\geq0$, and the characteristic exponents at $r=\infty$ are
\begin{equation}
-\frac{3}{2}\pm i\sqrt{\frac{7}{2}+\Lambda}.
\end{equation}
When $\omega\neq0$, the equation $\mathbf{R}R=0$ admits normal solutions (see \cite[Section~3.2]{MR0078494}), near $r=\infty$, of the asymptotic form
\begin{equation}
R\sim e^{\pm r\omega}r^{-1\pm2(M\omega-1)}.
\end{equation}
\end{lem}
\begin{proof}
The fact that $r=r_\pm$ are regular singular points follows directly from the fact that $\Delta=(r-r_+)(r-r_-)$, the statement about the type and rank of the singular point at $r=\infty$ follows directly from the discussion in \cite[Section~3.1]{MR0078494}, and the expressions for the characteristic exponents can be seen from the discussion in \cite[Section~1.1.3]{MR1858237}. Letting $R=y/\sqrt{\Delta}$, the equation $\mathbf{R}R=0$ is transformed into \begin{equation}\frac{d^2y}{dr^2}+qy=0,\end{equation} where \begin{equation}q(r)=\frac{U(r)}{\Delta}+\left(\frac{r_+-r_-}{2\Delta}\right)^2=-\omega^2-\frac{4\omega(M\omega-1)}{r}+O(r^{-2}).\end{equation} Following \cite[Section~3.2]{MR0078494}, the equation $\mathbf{R}R=0$ therefore has normal solutions of the asymptotic form \begin{equation}R\sim e^{\pm r\omega}r^{-1\pm2(M\omega-1)}.\end{equation}
\end{proof}
\begin{lem}
For a solution to the equation $\mathbf{R}R=0$ coming from a (globally smooth) perturbation of the Kerr metric, the corresponding characteristic exponent at $r=r_+$ is
\begin{equation}
\left|1+\frac{2Mr_++am}{r_+-r_-}\right|.
\end{equation}
\end{lem}
\begin{proof}
Since the set corresponding to $r=r_+$ is compact, and by assumption, the perturbation $\dot{W}$ of the Weyl tensor is continuous, $\dot{W}$ has bounded norm in a neighborhood of this set. Consequently, since $l$ and $m$ have norm $1$, it follows that $\dot{\Psi}_0=-\dot{W}(l,m,l,m)$ is bounded near $r=r_+$. Since $\Psi_2^{-2/3}=O(r^2)$, it follows that $\Phi$, and therefore $R$, is bounded near $r=r_+$. The statement now follows immediately.
\end{proof}
A perturbation $\dot{g}$ of the Kerr metric $g$ is said to be \emph{asymptotically flat} (AF) if $|\dot{g}|_g=O(r^{-1})$, with corresponding decay on derivatives, i.e.\ $|\nabla^k\dot{g}|_g=O(r^{-1-k})$ for every positive integer $k$. Here, $r$ is the radial coordinate of Kerr defined earlier, and the norm and covariant derivative are taken with respect to $g$.
\begin{lem}
Let $R$ be a solution to the equation $\mathbf{R}R=0$ coming from an asymptotically flat perturbation of the Kerr metric. When $\omega=0$, none of the characteristic exponents at $r=\infty$ are compatible with the asymptotic flatness assumption. When $\omega\neq 0$, exactly one of the asymptotic normal solutions is compatible with this assumption, namely
\begin{equation}
R\sim e^{-r|\omega|}r^{-1-2(M\omega-1)\operatorname{sgn}(\omega)}.
\end{equation}
\end{lem}
\begin{proof}
By the assumption of asymptotic flatness, we have $W=O(r^{-3})$, which means that $\Phi$, and therefore $R$, decays as $r^{-1}$ as $r\to\infty$. In particular, we must have $\lim_{r\to\infty}R(r)=0$, and the result now follows immediately.
\end{proof}
\subsection{Mode Stability}
Equipped with the lemmas of the previous subsection, we are now in a position to prove Theorem \ref{thm:riemannian_kerr_mode_stability}.
\begin{proof}[Proof of Theorem \ref{thm:riemannian_kerr_mode_stability}]
For $r>r_+$ and $-1<x<1$, note that \begin{align}\begin{split}U(r)+V(x)&=\begin{multlined}[t]-\frac{16 M (r+a x)}{(r-a x)^2}\\-\frac{(a^2 x (m x-2)+2 a (x^2-1) (M (r \omega -1)+r)+r (m-2 x) (2 M-r))^2}{(1-x^2)(\Delta+(1-x^2)a^2)\Delta}\\-\frac{(2 M (a x+3 r)+(r-a x) (r+a x) (-a x \omega +r \omega -2))^2}{(r-a x)^2(\Delta+(1-x^2)a^2)}\end{multlined}\\&<0.\end{split}\label{eq:riemannian_kerr_U_V_inequality}\end{align} Here, the strict negativity follows from that of the first term, which holds because $r_+>|a|$. By an integration of parts, we have \begin{align}U(r)&\leq U(r)+\int_0^\pi\left(\frac{dS}{d\theta}\right)^2\sin\theta\,d\theta=U(r)+\int_0^\pi\left(\mathbf{S}S-\frac{1}{\sin\theta}\frac{d}{d\theta}\left(\sin\theta\frac{dS}{d\theta}\right)\right)S\sin\theta\,d\theta\\&=U(r)+\int_0^\pi V(\cos\theta)S^2\sin\theta\,d\theta=\int_0^\pi(U(r)+V(\cos\theta))S^2\sin\theta\,d\theta<0,\end{align} where the last equality follows from \eqref{eq:riemannian_kerr_U_V_inequality} and the normalization of $S$. Multiplying \eqref{eq:riemannian_kerr_radial_equation} by $\overline{R}$ and integrating, the first term being integrated by parts, we get \begin{equation}0=\left[\Delta\frac{dR}{dr}\overline{R}\right]_{r=r_+}^{r=\infty}-\int_{r_+}^\infty\left(\Delta\left|\frac{dR}{dr}\right|^2-U|R|^2\right)dr.\end{equation} We claim that the first term vanishes; to see this, we consider the endpoints separately. Near $r=\infty$, we have $\Delta\sim r^2$, while $R$ and its derivative decays exponentially. Thus, the term in square brackets decays exponentially, and in particular, it goes to zero as $r\to\infty$. Near $r=r_+$, we know that $\overline{R}$ is bounded. The characteristic exponent corresponding to $R$ is either positive, in which case $\frac{dR}{dr}=o(r^{-1})$, or $R$ is analytic in a neighborhood of $r=r_+$, in which case $\frac{dR}{dr}$ is bounded. In either case, the product $\Delta\frac{dR}{dr}$ goes to zero as $r\to r_+$, and from this, it immediately follows that the term in square brackets goes to zero. 

We have thus shown that \begin{equation}\int_{r_+}^\infty\left(\Delta\left|\frac{dR}{dr}\right|^2-U|R|^2\right)\,dr=0.\end{equation} Since $U<0$, the terms in the integrand are both non-negative and must therefore vanish. We conclude that $R$ vanishes identically.
\end{proof}
\section{The Taub-Bolt Instanton}\label{sec:taub_bolt}
The general \emph{Taub--NUT} family of Ricci-flat metrics, depending on two parameters $M,N>0$, is given\footnote{Note that $\Delta$ has a different meaning in this section, unrelated to those given in Section \ref{sec:riemannian_NP_formalism} and Section \ref{sec:riemannian_kerr}. The current meaning of $\Delta$ will be retained throughout all of Section \ref{sec:taub_bolt} and Appendix \ref{sec:taub_bolt_spin_coefficients}.} in coordinates $(t,r,\theta,\phi)$ by \begin{equation}g=\frac{\Sigma}{\Delta}\,dr^2+4N^2\frac{\Delta}{\Sigma}(dt+\cos\theta\,d\phi)^2+\Sigma(d\theta^2+\sin^2\theta\,d\phi^2),\end{equation} where $\Delta=r^2-2Mr+N^2$ and $\Sigma=r^2-N^2$. Setting $M=N$ yields the self-dual Taub--NUT metric, a complete metric on $\mathbb{R}^4$.

Another metric of interest, the \emph{Taub-bolt metric}, arises by letting $M=\frac{5}{4}N$, which we shall do from now. We will now briefly account for the regularity of this metric. Introducing the coordinate system $(\tilde{t},\tilde{r},\tilde{\theta},\phi)$ by
\begin{equation}\begin{cases}r&=\frac{N}{4}(5+3\cosh\tilde{r}),\\t&=2\tilde{t}-\phi,\\\theta&=2\arctan(\frac{\tilde{\theta}}{2}),\end{cases}\end{equation} we see that $r$ is smooth as a function of $\tilde{r}^2$, and that \begin{equation}g=\Sigma(d\tilde{r}^2+(\tilde{r}^2+O(\tilde{r}^4))\,d\tilde{t}^2+(1+O(\tilde{\theta}^2))\,d\tilde{\theta}^2+(\tilde{\theta}^2+O(\tilde{\theta}^4))\,d\phi^2+O(\tilde{r}^2\tilde{\theta}^2)\,d\tilde{t}d\phi).\end{equation} Viewing $(\tilde{\theta},\phi)$ as polar coordinates on $\mathbb{R}^2\times\{y\}\subseteq\mathbb{R}^4$, and viewing $(\tilde{r},\tilde{t})$ as polar coordinates on $\{x\}\times\mathbb{R}^2\subseteq\mathbb{R}^4$, it follows that $g$ extends to a smooth metric on $\mathbb{R}^4\cong\mathbb{C}^2$, provided that we identify $\tilde{t}$ and $\phi$ with period $2\pi$ independently. This is equivalent to making the identifications $(t,\phi)\sim(t+4\pi,\phi)\sim(t+2\pi,\phi+2\pi)$.

We can also introduce another coordinate system $(\hat{t},\tilde{r},\hat{\theta},\phi)$ by \begin{equation}\begin{cases}t&=2\hat{t}+\phi,\\\theta&=2\arccot(\frac{\hat{\theta}}{2}),\end{cases}\end{equation} so that \begin{equation}g=\Sigma(d\tilde{r}^2+(\tilde{r}^2+O(\tilde{r}^4))\,d\hat{t}^2+(1+O(\hat{\theta}^2))\,d\hat{\theta}^2+(\hat{\theta}^2+O(\hat{\theta}^4))\,d\phi^2+O(\tilde{r}^2\hat{\theta}^2)\,d\hat{t}d\phi).\end{equation} In the same way as for the previous coordinate system, this shows that $g$ extends to a smooth metric on another copy of $\mathbb{C}^2$. Note that the identifications made to ensure regularity in the coordinate system $(\tilde{t},\tilde{r},\tilde{\theta},\phi)$ also ensure regularity in the coordinate system $(\hat{t},\tilde{r},\hat{\theta},\phi)$. When defined on the union of these copies of $\mathbb{C}^2$, this metric is complete.

Computing the transition map between the two coordinate systems, we see that they are related by $(\hat{t},\tilde{r},\hat{\theta},\phi)=(\tilde{t}-\phi,\tilde{r},\frac{4}{\tilde{\theta}},\phi)$. In other words, the two copies of $\mathbb{C}^2$ are glued together according to the map \begin{align*}(\mathbb{C}\setminus\{0\})\times\mathbb{C}&\to(\mathbb{C}\setminus\{0\})\times\mathbb{C},\\(z_1,z_2)&\mapsto\left(\frac{4}{\overline{z_1}},z_2\cdot\frac{|z_1|}{z_1}\right).\end{align*} Topologically, this is the same thing as gluing two such copies along the map $(z_1,z_2)\mapsto(\frac{1}{z_1},z_2\cdot\frac{|z_1|}{z_1})$, or equivalently, gluing together two copies of $\overline{D}^2\times\mathbb{C}$ along the map \begin{align}\begin{split}S^1\times\mathbb{C}&\to S^1\times\mathbb{C},\\(z_1,z_2)&\mapsto\left(\frac{1}{z_1},\frac{z_2}{z_1}\right).\end{split}\label{eq:taub_bolt_gluing_map}\end{align}

We now claim that the manifold is diffeomorphic to $\mathbb{C}P^2$ minus a point. To see this, consider two of the projective coordinate charts for $\mathbb{C}P^2$, $(U_0,\phi_0)$ and $(U_1,\phi_1)$, where \begin{equation}U_i=\{[Z_0:Z_1:Z_2]\in\mathbb{C}P^2\mid Z_i\neq 0\},\end{equation} and \begin{align*}\phi_0:U_0&\to\mathbb{C}^2,\\ [Z_0:Z_1:Z_2]&\mapsto\left(\frac{Z_1}{Z_0},\frac{Z_2}{Z_0}\right),\\\phi_1:U_1&\to\mathbb{C}^2,\\ [Z_0:Z_1:Z_2]&\mapsto\left(\frac{Z_0}{Z_1},\frac{Z_2}{Z_1}\right).\end{align*} Since $U_0\cup U_1=\mathbb{C}P^2\setminus\{[0:0:1]\}$, it follows that the latter is topologically equivalent to two copies of $\mathbb{C}^2$, glued together along the transition map \begin{align*}(\mathbb{C}\setminus\{0\})\times\mathbb{C}&\to(\mathbb{C}\setminus\{0\})\times\mathbb{C},\\(z_1,z_2)&\mapsto\left(\frac{1}{z_1},\frac{z_2}{z_1}\right).\end{align*} Again, topologically this is the same thing as gluing together two copies of $\overline{D}^2\times\mathbb{C}$ along the map \eqref{eq:taub_bolt_gluing_map}. This shows that the manifold is \emph{homeomorphic} to $\mathbb{C}P^2$ minus a point. To show that these are \emph{diffeomorphic}, we can replace the closed disk with an open disk of radius slightly larger than $1$, gluing the two spaces together along a thin open strip around $S^1$. The gluing map will then be isotopic to the corresponding transition map in $\mathbb{C}P^2$.
\subsection{The Separated Perturbation Equations in Coordinates}\label{sec:taub_bolt_perturbation_equation}
As for Kerr, we are interested in a particular choice of complex null tetrad $(l,\overline{l},m,\overline{m})$, in this case given by
\begin{align}l&=\frac{1}{\sqrt{2\Sigma}}\left(\frac{1}{\sin\theta}\left(\cos\theta\frac{\partial}{\partial t}-\frac{\partial}{\partial\phi}\right)+i\frac{\partial}{\partial\theta}\right),\label{eq:taub_bolt_tetrad_l}\\
m & =\sqrt{\frac{\Delta}{2\Sigma}}\frac{\partial}{\partial r}+\frac{i\sqrt{\Sigma/2\Delta}}{2N}\frac{\partial}{\partial t},\label{eq:taub_bolt_tetrad_m} 
\end{align}
satisfying $|l|_g=|m|_g=1$.

The spin coefficients for this tetrad are given explicitly in Section \ref{sec:taub_bolt_spin_coefficients}. For this tetrad, we have \begin{equation}\Psi_2=\frac{N}{4(r-N)^3},\qquad\tilde{\Psi}_2=\frac{9N}{4(r+N)^3},\end{equation} and the rest of the Weyl scalars vanish. Thus, this is an adapted tetrad, and we see that the Taub-bolt metric is of type D.

The following four lemmas give the relevant properties of the perturbation equations \eqref{eq:riemannian_teukolsky_1} \eqref{eq:riemannian_teukolsky_2} for our analysis. The proofs are entirely analogous to those in Section \ref{sec:riemannian_kerr_perturbation_equation} and are therefore omitted.
\begin{lem}
For the tetrad given in \eqref{eq:taub_bolt_tetrad_l} and \eqref{eq:taub_bolt_tetrad_m}, the perturbation equation \eqref{eq:riemannian_teukolsky_1} is equivalent to the equation $\mathbf{L}\Phi=0$, where $\Phi=\Psi_2^{-2/3}\dot{\Psi}_0$ and
\begin{multline}
\mathbf{L}=\frac{\partial}{\partial r}\Delta\frac{\partial}{\partial r}-\frac{4N(r+N)}{(r-N)^2}+\frac{\Sigma^2}{4N^2\Delta}\left(\frac{\partial}{\partial t}-i\frac{N(4r^2-11Nr+3N^2)}{\Sigma(r-N)}\right)^2\\+\frac{1}{\sin\theta}\frac{\partial}{\partial\theta}\sin\theta\frac{\partial}{\partial\theta}+\frac{1}{\sin^2\theta}\left(\cos\theta\frac{\partial}{\partial t}-\frac{\partial}{\partial\phi}-2i\cos\theta\right)^2.
\end{multline}

Furthermore, if $\Phi$ is a solution to this equation coming from a perturbation of the metric, then we can write
\begin{equation}
\Phi(t,r,\theta,\phi)=\sum_{m,\omega,\Lambda}e^{i(m\phi-\omega t)}R_{m,\omega,\Lambda}(r)S_{m,\omega,\Lambda}(\theta),
\end{equation}
where $m$ runs over $\frac{1}{2}\mathbb{Z}$, and $\omega$ runs over $m+\mathbb{Z}$, and for each choice of $m,\omega,\Lambda$, the function $R=R_{m,\omega,\Lambda}$ solves the equation $\mathbf{R}R=0$, and the function $S=S_{m,\omega,\Lambda}$ solves the boundary value problem $\mathbf{S}S=0$, $S'(0)=S'(\pi)=0$, where \begin{equation}\mathbf{R}=\frac{d}{dr}\Delta\frac{d}{dr}+U(r),\end{equation}\begin{equation}U(r)=-\frac{4N(r+N)}{(r-N)^2}-\frac{\Sigma^2}{4N^2\Delta}\left(\omega+\frac{N(4r^2-11Nr+3N^2)}{\Sigma(r-N)}\right)^2-\Lambda\end{equation} and \begin{equation}\mathbf{S}=\frac{1}{\sin\theta}\frac{d}{d\theta}\sin\theta\frac{d}{d\theta}+V(\cos\theta),\end{equation}\begin{equation}V(x)=-\frac{((\omega+2)x+m)^2}{1-x^2}+\Lambda.\end{equation} The separation constant $\Lambda$ runs over the (countable set of) values for which such an $S$ exists, all of which are non-negative.

The same statement holds if $\Phi$ is replaced by $\tilde{\Phi}=\tilde{\Psi}_2^{-2/3}\dot{\tilde{\Psi}}_0$, the operator $\mathbf{L}$ is replaced by $\tilde{\mathbf{L}}$, and the operator $\mathbf{R}$ replaced by $\tilde{\mathbf{R}}$, defined in the same way but using a potential $\tilde{U}$ in place of $U$. Here,
\begin{multline}
\tilde{\mathbf{L}}=\frac{\partial}{\partial r}\Delta\frac{\partial}{\partial r}-\frac{36N(r-N)}{(r+N)^2}+\frac{\Sigma^2}{4N^2\Delta}\left(\frac{\partial}{\partial t}+i\frac{N(4r^2-19Nr+13N^2)}{\Sigma(r+N)}\right)^2\\+\frac{1}{\sin\theta}\frac{\partial}{\partial\theta}\sin\theta\frac{\partial}{\partial\theta}+\frac{1}{\sin^2\theta}\left(\cos\theta\frac{\partial}{\partial t}-\frac{\partial}{\partial\phi}-2i\cos\theta\right)^2
\end{multline}
and
\begin{equation}\tilde{U}(r)=-\frac{36N(r-N)}{(r+N)^2}-\frac{\Sigma^2}{4N^2\Delta}\left(\omega-\frac{N(4r^2-19Nr+13N^2)}{\Sigma(r+N)}\right)^2-\Lambda.\end{equation}
\end{lem}
\begin{lem}
The equation $\mathbf{R}R=0$ is an ordinary differential equation in a complex variable $r$, which has regular singular points at $r=2N$ and $r=N/2$. The point $r=\infty$ is an irregular singular point of rank $1$, except when $\omega=0$, in which case it is a regular singular point. Thus, the equation $\mathbf{R}R=0$ is a confluent Heun equation (see \cite[Section~3]{MR1858237}) when $\omega\neq 0$, and a hypergeometric equation (see \cite[Section~2]{MR1858237}) when $\omega=0$. The characteristic exponents at $r=2N$ are
\begin{equation}
\pm(\omega-1),
\end{equation} and those at $r=N/2$ are
\begin{equation}
\pm\left(\frac{\omega}{4}-1\right).
\end{equation}
When $\omega=0$, the characteristic exponents at $r=\infty$ are
\begin{equation}
-\frac{3}{2}\pm i\sqrt{\frac{7}{2}+\Lambda}.
\end{equation}
When $\omega\neq0$, the equation $\mathbf{R}R=0$ admits normal solutions (see \cite[Section~3.2]{MR0078494}), near $r=\infty$, of the asymptotic form
\begin{equation}
R\sim e^{\pm r\omega/2N}r^{-1\pm(5\omega/4-2)}.
\end{equation}
\end{lem}
\begin{lem}
For a solution to the equation $\mathbf{R}R=0$ coming from a (globally smooth) perturbation of the Taub-bolt metric, the corresponding characteristic exponent at $r=2N$ is $|\omega-1|$.
\end{lem}
A perturbation $\dot{g}$ of the Taub-bolt metric is said to be \emph{asymptotically locally flat} (ALF) if it decays as $O(r^{-1})$, with corresponding decay on derivatives, just like the definition of AF perturbations of the Kerr metric given in Section \ref{sec:riemannian_kerr_perturbation_equation}.
\begin{lem}
Let $R$ be a solution to the equation $\mathbf{R}R=0$ coming from an asymptotically locally flat perturbation of the Taub-bolt metric. When $\omega=0$, none of the characteristic exponents at $r=\infty$ are compatible with the assumption of asymptotic local flatness. When $\omega\neq 0$, exactly one of the asymptotic normal solutions is compatible with this assumption, namely
\begin{equation}
R\sim e^{-r|\omega|/2N}r^{-1-(5\omega/4-2)\operatorname{sgn}(\omega)}.
\end{equation}
\end{lem}
Corresponding lemmas regarding the asymptotics of the equation $\tilde{\mathbf{R}}R=0$ also hold. We omit them since they are entirely analogous.
\subsection{Mode Stability}
\begin{proof}[Proof of Theorem \ref{thm:taub_bolt_mode_stability}]
In this case, we directly see that $U(r)<0$, and integrating the equation $\mathbf{R}R=0$ by parts like in the proof of Theorem \ref{thm:riemannian_kerr_mode_stability}, we see that $R$ vanishes identically. The case involving the equation $\tilde{\mathbf{R}}R=0$ is entirely analogous.
\end{proof}
\appendix
\section{Newman--Penrose Equations}\label{sec:NP_equations}
Given an equation expressed in terms of the spin coefficients, Weyl scalars and tetrad derivative operators, we can apply the \emph{tilde operation}, given by formally replacing any such quantity $x$ by the tilded quantity $\tilde{x}$. Here, we adopt the convention that $\tilde{\tilde{x}}=x$, $\tilde{D}=D$ and $\tilde{\Delta}=\Delta$. The result is a new, a priori independent equation, the \emph{tilded version} of the original equation.

We have the Newman--Penrose commutation relations, given by the four equations
\begin{align}
[{D},{\Delta}]\psi &= - (\gamma + \tilde{\gamma}) {D}\psi               
-  (\epsilon + \tilde{\epsilon}) {\Delta}\psi
+ (\pi + \tilde{\tau}) {\delta}\psi
+ (\tilde{\pi} + \tau) {\tilde{\delta}}\psi ,\\
\label{eq:commutator_D_delta}
[{D},{\delta}]\psi &= - (\tilde{\alpha} + \beta -  \tilde{\pi}) {D}\psi 
-  \kappa {\Delta}\psi
+ (\epsilon -  \tilde{\epsilon} + \tilde{\rho}) {\delta}\psi
+ \sigma {\tilde{\delta}}\psi ,\\
[{\Delta},{\delta}]\psi &= \tilde{\nu} {D}\psi                               
+ (\tilde{\alpha} + \beta -  \tau) {\Delta}\psi
+ (\gamma -  \tilde{\gamma} -  \mu) {\delta}\psi
-  \tilde{\lambda} {\tilde{\delta}}\psi ,\\
[{\delta},{\tilde{\delta}}]\psi &= (\mu -  \tilde{\mu}) {D}\psi                      
+ (\rho -  \tilde{\rho}) {\Delta}\psi
+ (- \alpha + \tilde{\beta}) {\delta}\psi
+ (\tilde{\alpha} -  \beta) {\tilde{\delta}}\psi,
\end{align}
together with their tilded versions.\footnote{For the first and last equation, applying the tilde operation results in the same equations, up to sign. For the second and third equations, it results in two new, independent equations.}

In terms of the spin coefficients, the vacuum Einstein equations become
\begin{align}
{D}\alpha -  {\tilde{\delta}}\epsilon ={}&- \tilde{\beta} \epsilon
 -  \gamma \tilde{\kappa}
 -  \kappa \lambda
 + \pi (\epsilon + \rho)
 + \alpha (-2 \epsilon + \tilde{\epsilon} + \rho)
 + \beta \tilde{\sigma},\\
\label{eq:D_beta}
{D}\beta -  {\delta}\epsilon ={}&\Psi_{1}{}
 -  \tilde{\alpha} \epsilon
 -  \kappa (\gamma + \mu)
 + \epsilon \tilde{\pi}
 + \beta (- \tilde{\epsilon} + \tilde{\rho})
 + (\alpha + \pi) \sigma ,\\
{D}\gamma -  {\Delta}\epsilon ={}&\Psi_{2}{}
 -  \tilde{\gamma} \epsilon
 -  \gamma (2 \epsilon + \tilde{\epsilon})
 -  \kappa \nu
 + \pi (\beta + \tau)
 + \alpha (\tilde{\pi} + \tau)
 + \beta \tilde{\tau},\\
{D}\lambda -  {\tilde{\delta}}\pi ={}&- \tilde{\kappa} \nu
 + \pi (\alpha -  \tilde{\beta} + \pi)
 + \lambda (-3 \epsilon + \tilde{\epsilon} + \rho)
 + \mu \tilde{\sigma},\\
{D}\rho -  {\tilde{\delta}}\kappa ={}&\kappa (-3 \alpha -  \tilde{\beta} + \pi)
 + \rho (\epsilon + \tilde{\epsilon} + \rho)
 + \sigma \tilde{\sigma}
 -  \tilde{\kappa} \tau ,\\
\label{eq:D_sigma}
{D}\sigma -  {\delta}\kappa ={}&\Psi_{0}{}
 + (3 \epsilon -  \tilde{\epsilon} + \rho + \tilde{\rho}) \sigma
 -  \kappa (\tilde{\alpha} + 3 \beta -  \tilde{\pi} + \tau),\\
\label{eq:D_tau}
{D}\tau -  {\Delta}\kappa ={}&\Psi_{1}{}
 -  (3 \gamma + \tilde{\gamma}) \kappa
 + \tilde{\pi} \rho
 + (\epsilon -  \tilde{\epsilon} + \rho) \tau
 + \sigma (\pi + \tilde{\tau}),\\
{\Delta}\rho -  {\tilde{\delta}}\tau ={}&- \Psi_{2}{}
 + \kappa \nu
 + (\gamma + \tilde{\gamma} -  \tilde{\mu}) \rho
 -  \lambda \sigma
 -  \tau (\alpha -  \tilde{\beta} + \tilde{\tau}),\\
{\delta}\alpha -  {\tilde{\delta}}\beta ={}&- \Psi_{2}{}
 + \alpha (\tilde{\alpha} - 2 \beta)
 + \beta \tilde{\beta}
 + \epsilon (\mu -  \tilde{\mu})
 + (\gamma + \mu) \rho
 -  \gamma \tilde{\rho}
 -  \lambda \sigma ,\\
\label{eq:delta_rho}
{\delta}\rho -  {\tilde{\delta}}\sigma ={}&- \Psi_{1}{}
 + \kappa (\mu -  \tilde{\mu})
 + (\tilde{\alpha} + \beta) \rho
 + (-3 \alpha + \tilde{\beta}) \sigma
 + (\rho -  \tilde{\rho}) \tau .
\end{align}
together with their tilded versions.

Finally, we have the Bianchi identities, given by the equations
\begin{align}
\label{eq:deltatPsi0}
{\tilde{\delta}}\Psi_{0} -  {D}\Psi_{1}&= (4 \alpha -  \pi) \Psi_{0}                         
- 2 (\epsilon + 2 \rho) \Psi_{1}
+ 3 \kappa \Psi_{2},\\
\label{eq:deltat_Psi1}
{\tilde{\delta}}\Psi_{1} -  {D}\Psi_{2}&= \lambda \Psi_{0}                                   
+ 2 (\alpha -  \pi) \Psi_{1}
- 3 \rho \Psi_{2}
+ 2 \kappa \Psi_{3}{},\\
\label{eq:DeltaPsi0}
{\Delta}\Psi_{0} -  {\delta}\Psi_{1}&= 4 \gamma \Psi_{0}                                  
-  \mu \Psi_{0}
- 2 (\beta + 2 \tau) \Psi_{1}
+ 3 \sigma \Psi_{2},\\
\label{eq:Delta_Psi1}
{\Delta}\Psi_{1} -  {\delta}\Psi_{2}&= \nu \Psi_{0}                                       
+ 2 (\gamma -  \mu) \Psi_{1}
- 3 \tau \Psi_{2}
+ 2 \sigma \Psi_{3}{},
\end{align}
together with their tilded versions.
\section{Spin Coefficients}
\subsection{Riemannian Kerr}\label{sec:riemannian_kerr_spin_coefficients}
\begin{align}
\alpha         & =\beta =\frac{r\cos\theta-a}{(r-a\cos\theta)2\sqrt{2\Sigma}\sin\theta},           \\
\gamma         & =\epsilon =\frac{i(-\Delta/(r-a\cos\theta)+r-M)}{2\sqrt{2\Delta\Sigma}},          \\
\mu            & =\rho =-\frac{i\sqrt{\Delta/2\Sigma}}{r-a\cos\theta},                             \\
\pi            & =\tau =-\frac{a\sin\theta}{(r-a\cos\theta)\sqrt{2\Sigma}},                        \\
\tilde{\alpha} & =\tilde{\beta}=-\frac{r\cos\theta+a}{(r+a\cos\theta)2\sqrt{2\Sigma}\sin\theta},   \\
\tilde{\gamma} & =\tilde{\epsilon}=\frac{i(-\Delta/(r+a\cos\theta)+r-M)}{2\sqrt{2\Delta\Sigma}},   \\
\tilde{\mu}    & =\tilde{\rho}=- \frac{i \sqrt{\Delta/2\Sigma}}{r+a \cos\theta},                   \\
\tilde{\pi}    & =\tilde{\tau}=- \frac{a \sin\theta}{(r+a \cos\theta) \sqrt{2\Sigma}},             \\
\kappa         & =\lambda=\nu=\sigma=\tilde{\kappa}=\tilde{\lambda}=\tilde{\nu}=\tilde{\sigma} =0.
\end{align}
\subsection{Taub-Bolt}\label{sec:taub_bolt_spin_coefficients}
\begin{align}
\alpha         & =\beta =\frac{N(r+N)^2}{8\Sigma\sqrt{2\Delta\Sigma}},                                                              \\
\gamma         & =\epsilon =\tilde{\gamma} =\tilde{\epsilon} =\frac{i\cot\theta}{2\sqrt{2\Sigma}},                                                                    \\
\pi            & =\tau=-\frac{(r+N)\sqrt{\Delta/2\Sigma}}{\Sigma},                                                                  \\
\tilde{\alpha} & =\tilde{\beta}=-\frac{9N(r-N)}{8(r+N)\sqrt{\Delta\Sigma}},                                                         \\
\tilde{\pi}    & =\tilde{\tau}=\frac{\sqrt{\Delta/2\Sigma}}{r+N},                                                                   \\
\kappa         & =\lambda=\mu=\nu=\rho=\sigma=\tilde{\kappa}=\tilde{\lambda}=\tilde{\mu}=\tilde{\nu}=\tilde{\rho}=\tilde{\sigma}=0.
\end{align}
\bibliographystyle{plain}
\bibliography{main}
\end{document}